\documentclass{article}
\usepackage[utf8]{inputenc}
\usepackage{amsthm}
\usepackage{amstext}
\usepackage{amsmath}
\usepackage{amssymb}
\usepackage[all]{xy}
\usepackage{color}

\usepackage{authblk}

\usepackage{enumitem}

\newtheorem{definition}{Def\text{}inition}[section]
\newtheorem{theorem}[definition]{Theorem}
\newtheorem{example}[definition]{Example}

\newtheorem{lemma}[definition]{Lemma}
\newtheorem{proposition}[definition]{Proposition}
\newtheorem{corollary}[definition]{Corollary}

\newtheorem{question}[definition]{Question}

\newtheorem{observation}[definition]{Observation}

\begin{document}

\title{ \bf \large Weak normality properties in $\Psi$-spaces}
\author{ \small SERGIO A. GARCIA-BALAN AND PAUL J. SZEPTYCKI}
\date{}
\maketitle

\begin{abstract} 
\noindent Almost disjoint families of true cardinality $\mathfrak{c}$ are used to produce an example of a mildly-normal not partly-normal $\Psi$-space and a quasi-normal not almost-normal $\Psi$-space. This is related with a problem posed by Kalantan in \cite{K2017} where he asks whether there exists a mad family so that the related Mr\'owka-Isbell space is partly-normal. In addition, a consistent example of a Luzin mad family such that its associated $\Psi$-space is quasi-normal is provided.
\end{abstract}

\let\thefootnote\relax\footnote{\today }
\let\thefootnote\relax
\let\thefootnote\relax

\section{Introduction}

\noindent Mr\'owka-Isbell $\Psi$-spaces give a number of interesting counterexamples in many areas of topology including normality and related covering properties  (\cite{Mr}, \cite{GJ}). $\Psi$-spaces associated to maximal almost disjoint families are never normal. Weakenings of normality have been considered in the literaure since the late 60's and early 70's . For instance, quasi-normal \cite{Za}, almost-normal \cite{AS}, mildly-normal \cite{Shch}, \cite{SS}, and more recently $\pi$-normal \cite{K2008} and partly-normal  \cite{K2017}. In \cite{KS} L. Kalantan and the second author prove that any product of ordinals is mildly-normal. Kalantan builds a $\Psi$-space which is not mildly-normal in \cite{K2002} and, in \cite{K2017}, using {\bf CH} constructs a mad family so that the associated $\Psi$-space is quasi-normal.\\

\noindent Standard notation is followed and any undefined term can be found in \cite{E}. A subset $A$ of a topological space $X$ is called \emph{regularly closed} (also called closed domain), if $A = \overline{int(A)}$ ($cl_X(A)$ or simply $cl(A)$ will denote the closure of $A$ in the space $X$ as well). A set $A$ will be called \emph{$\pi$-closed}, if $A$ is a finite intersection of regularly closed sets. Two subsets $A$ and $B$ of a topological space $X$ are said to be \emph{separated} if there exist two disjoint open sets $U$ and $V$ of $X$ such that $A\subseteq U$ and $B \subseteq V$. 

\newpage

\begin{definition}A regular space $X$ is called:
\begin{enumerate}
\item \emph{$\pi$-normal \cite{K2008}} if any two nonintersecting sets $A$ and $B$, where $A$ is closed and $B$ is $\pi$-closed, are separated.
\item \emph{almost-normal \cite{AS}} if any two nonintersecting sets $A$ and $B$, where $A$ is closed and $B$ is regularly closed, are separated.
\item \emph{quasi-normal \cite{Za}} if any two nonintersecting $\pi$-closed sets $A$ and $B$ are separated.
\item \emph{partly-normal \cite{K2017}} if any two nonintersecting sets $A$ and $B$, where $A$ is regular closed and $B$ is $\pi$-closed, are separated.
\item \emph{mildly-normal (also called $\kappa$-normal), \cite{Shch} \cite{SS}} if any two nonintersecting regular closed sets $A$ and $B$ are separated.
\end{enumerate}
\end{definition}

\noindent Since ``regular closed $\rightarrow$ $\pi$-closed  $\rightarrow$ closed'' holds, it follows that normal spaces are $\pi$-normal and:  
\begin{displaymath}
 \begin{array}{ccccccccc}
                                 & \nearrow & \textit{quasi-normal} & \searrow &                                   &                  &  \\
 \textit{$\pi$-normal} &               &                                   &               & \textit{partly-normal} & \rightarrow & \textit{mildly-normal.} \\
                                  & \searrow & \textit{almost-normal} & \nearrow &                                 &                    &  
\end{array} 
\end{displaymath}

\begin{proposition}
\label{PiandAlmostNormalCoincide}
Almost-normal spaces are $\pi$-normal.
\end{proposition}

\begin{proof}
Assume $X$ is an almost-normal space. For a positive integer $n$, call a set $n$-$\pi$-closed, if it is the intersection of $n$ many regular closed sets. We will show by induction on $n$, that in $X$ every $n$-$\pi$-closed set can be separated from a closed set, provided they are disjoint. This is enough to show that $X$ is $\pi$-normal.\\
\emph{Base case:} $n=1$. Since $X$ is almost normal, every closed $H$ and $1$-$\pi$-closed set $K$ in $X$ such that $H \cap K= \emptyset$ can be separated ($K$ is a regular closed set).
\emph{Inductive step:} Assume that for all $1\le i \le n$ if $H$  is closed, $K$ is $i$-$\pi$-closed in $X$ and, $H \cap K= \emptyset$, then $H$ and $K$ can be separated. Let $H \subset X$ be a closed set and let $K$ be an $(n+1)$-$\pi$-closed set such that $H \cap K = \emptyset$. Thus, $K = \bigcap _{0\le j \le n}K_j$, where each $K_j$ is a regular closed set in $X$. We show that $H$ and $K$ can be separated.\\

\noindent \emph{Case 1:} $H \cap (\bigcap _{j < n}K_j) = \emptyset$ ($H \cap K_n = \emptyset$).\\ 
Then, by the inductive hypothesis, we can find $U, V \subset X$ open such that $U \cap V =\emptyset$, $H \subseteq U$, $\bigcap _{j < n}K_j \subseteq V$ ($K_n \subseteq V$). Since $K \subseteq \bigcap _{j < n}K_j$ ($K \subseteq K_n$), $H$ and $K$ are separated by $U$ and $V$.\\

\noindent \emph{Case 2:} $H \cap (\bigcap _{j < n}K_j)\neq \emptyset \neq H \cap K_n$.\\
Given that $H \cap K = \emptyset$, $[H \cap (\bigcap _{j < n}K_j)] \cap K_n = \emptyset$. In addition, $H \cap (\bigcap _{j < n}K_j)$ is closed, non-empty and $K_n$ is a regular closed set, since $X$ is almost-normal, there are $U_n, V_n \subset X$ open such that $U_n \cap V_n = \emptyset$, $H \cap (\bigcap _{j < n}K_j) \subseteq U_n$, $K_n \subseteq V_n$.\\
\noindent Now, $H \smallsetminus U_n = H \cap (X \smallsetminus U_n)$ is closed, non-empty (since $H \cap K_n \subseteq H \smallsetminus U_n)$, and disjoint from $\bigcap _{j < n}K_j$, which is an $n$-$\pi$-closed set. Hence, by the inductive hypothesis, there are $U_K, V_K \subset X$ open such that $U_K \cap V_K = \emptyset$, $H \smallsetminus U_n \subseteq U_K$, $\bigcap _{j < n}K_j \subseteq V_K$. Let $U = U_n \cup U_K$, $V = V_n \cap V_K$.\\

\noindent \emph{Claim:} $U$ and $V$ are a separation of $H$ and $K$. \\
Assume there is $x \in U \cap V$, then $x \in U_n \cap V_n$ or $x \in U_K \cap V_K$, which is a contradiction. Thus, $U \cap V = \emptyset$. In addition, $H = (H \cap U_n) \cup (H \smallsetminus U_n) \subseteq U_n \cup U_K = U$ and $K = (\bigcap _{j < n}K_j) \cap K_n \subseteq V_K \cap V_n = V$. Hence, $H$ and $K$ are separated.\\
Therefore, for any closed set $H$ and for each $n$, if $K$ is $n$-$\pi$-closed and $H \cap K = \emptyset$, then $H$ and $K$ can be separated. Whence, $X$ is $\pi$-normal.
\end{proof}

\noindent Hence, the previous diagram is simplified as follows:
\begin{center}
\emph{almost-normal $\rightarrow$ quasi-normal $\rightarrow$ partly normal $\rightarrow$ mildly-normal}.
\end{center}

\noindent A family $\mathcal{A}$ of infinite subsets of $\omega$ is called an \emph{almost disjoint family} if and only if any two distinct members meet in a finite set (for each $a,b \in \mathcal{A}$, $a\neq b \rightarrow |a\cap b| < \omega$). All almost disjoint families considered here will be infinite.

\noindent Given an almost disjoint family $\mathcal{A}$, the Mr\'owka-Isbell $\Psi$-space $\Psi(\mathcal{A})$ is defined as follows: the underlying set is $\omega \cup \mathcal{A}$; if $n \in \omega$, $\{n\}$ is open and if $a\in \mathcal{A}$, then for any finite set $F \subset \omega$, $\{a\} \cup a \setminus F$ is a basic open set of $a$. $\Psi(\mathcal{A})$ is a separable, first countable, zero dimensional regular space. For a detailed survey on open problems and recent work on almost disjoint families and $\Psi$-spaces see \cite{HH}.

\begin{definition} Given an almost disjoint family $\mathcal{A}$,
\begin{itemize}[nosep]
\item If $B \subseteq \omega$, let $\mathcal{A}\upharpoonright _ B = \{a\in \mathcal{A}: |a \cap B| = \omega\}$.
\item $\mathcal{I}^+(\mathcal{A}) = \{B \subseteq \omega: |\mathcal{A}\upharpoonright _ B|\geq \omega\}$ is the family of big sets (the sets that have infinite intersection with infinite many members of the family).
\item $\mathcal{I}(\mathcal{A}) = \{B \subseteq \omega: |\mathcal{A}\upharpoonright _ B|< \omega\}$, the family of small sets. This family forms an ideal. 
\item $\mathcal{A}$ will be called \emph{completely separable \cite{GS}} if for each $B \in \mathcal{I}^+(\mathcal{A})$, there is some $a \in \mathcal{A}$ with $a \subseteq B$.
\item $\mathcal{A}$ will be called \emph{of true cardinality $\mathfrak{c}$ \cite{HS}} if for every $B \subseteq \omega$ either $\mathcal{A}\upharpoonright _ B $ is finite, or it has size $\mathfrak{c}$.
\item If $\Psi(\mathcal{A})$ is a normal space (almost-normal, quasi-normal, partly-normal, mildly-normal), it will be said that $\mathcal{A}$ is normal (almost-normal, quasi-normal, partly-normal, mildly-normal, respectively).
\end{itemize}
\end{definition}

\noindent The existence in ZFC of a completely separable mad family is an important open question that has many interesting consequences (see \cite{HS}). Completely separable almost disjoint (not maximal) families do exist in ZFC and also have interesting consequences (see \cite{GS}). It is not hard to show that if $\mathcal{A}$ is a completely separable almost disjoint family and $B \in \mathcal{I}^+(\mathcal{A})$, then $|\{a \in \mathcal{A}: a \subseteq {B}\}|= \mathfrak{c}$. This fact has the following consequence: if $\mathcal{A}$ is completely separable, then for any $B \subseteq \omega$, the set $\mathcal{A}\upharpoonright _ B$ is either finite or it has size $\mathfrak{c}$. That is, every completely separable almost disjoint family is of true cardinality $\mathfrak{c}$, and therefore, almost disjoint families of true cardinality $\mathfrak{c}$ exist in ZFC. Furthermore, every infinite almost disjoint family $\mathcal{A}$ of true cardinality $\mathfrak{c}$, has size $\mathfrak{c}$ and thefore $\mathcal{A}$ is not normal (as a consequence of Jones' Lemma). Actually, something slightly stronger holds:

\begin{observation}
\label{IfTrueCardcAandCtlbeComplementcannotseparate}
If $\mathcal{A}$ is an almost disjoint family of true cardinality $\mathfrak{c}$, then for all $\mathcal{C} \in [\mathcal{A}]^{\aleph _0}$, $\mathcal{C}$ and $\mathcal{A} \smallsetminus \mathcal{C}$ cannot be separated in $\Psi(\mathcal{A})$.
\end{observation}

\begin{proof}
Let $U$, $V$ be any open sets in $\Psi(\mathcal{A})$ so that $\mathcal{C} \subseteq U$, $\mathcal{A} \setminus \mathcal{C} \subseteq V$. Let $W = U \cap \omega$, then for all $c \in \mathcal{C}$, $c \subseteq ^* W$. Hence, $|\mathcal{A}\upharpoonright _ W|\geq \omega$. Thus,  $|\mathcal{A}\upharpoonright _ W|= \mathfrak{c}$. Pick $a \in \mathcal{A}\setminus \mathcal{C}$ such that $|W\cap a| = \omega$. Since $a \subseteq ^* V \cap \omega$, $U \cap V \neq \emptyset$.
\end{proof}

\noindent The following observations are not hard to show and they will be used in various occasions in the next section.

\begin{observation}
\label{closureofsubsetsofomega}
Given any almost disjoint family $\mathcal{A}$, if $W \subseteq \omega$, then  $cl_{\Psi(\mathcal{A})}(W)$ is a regular closed subset of $\Psi(\mathcal{A})$.
\end{observation}

\begin{observation}
\label{regclosedset}
Given any almost disjoint family $\mathcal{A}$, if $H \subset \Psi(\mathcal{A})$ is a regular closed set,
then for each $a \in \mathcal{A}$, $a \in H$ if and only if $|a \cap H| = \omega$.
\end{observation}

\begin{observation}
\label{closedsets}
Given any almost disjoint family $\mathcal{A}$ and $H, K \subset \Psi(\mathcal{A})$ such that $H$ and $K$ are closed sets, $H\cap K = \emptyset$ and $|H  \cap \mathcal{A}| < \omega$, then $H$ and $K$ can be separated. In particular, for each closed set $H \subset \Psi(\mathcal{A})$ that has finite intersection with $\mathcal{A}$, $H$ and $\mathcal{A} \smallsetminus H$ can be separated.
\end{observation}

\section{Examples}

\noindent Example \ref{QuasiNormalNotAlmostNormalPsiSpace} provides a quasi-normal not almost-normal almost disjoint family $\mathcal{F}$ which is constructed from a particular non almost-normal almost disjoint family  $\mathcal{A}$ of true cardinality $\mathfrak{c}$. Each element of $\mathcal{F}$ will be a finite union of elements of $\mathcal{A}$ . In order to make $\mathcal{F}$ quasi-normal, all pairs of disjoint $\pi$-closed sets in $\Psi(\mathcal{F})$ have to be separated. By Observation \ref{closedsets}, the only pairs of  $\pi$-closed sets $(A,B)$ that might be difficult to separate are the ones where $A  \cap \mathcal{F}$ and $B  \cap \mathcal{F}$ are infinite. Using that $\mathcal{A}$ is of true cardinality $\mathfrak{c}$ it will be possible to build $\mathcal{F}$ so that all such pairs have a point in common. Thus, all pairs of disjoint $\pi$-closed sets in $\Psi(\mathcal{F})$ will be trivial, i.e. one of them will have finite intersection with $\mathcal{F}$. Hence, $\mathcal{F}$ will be quasi-normal. In addition, it won't be hard to carry this construction out so that the non almost-normality of $\mathcal{A}$ is preserved in $\mathcal{F}$. That is, a closed set $\mathcal{C}$ and a regular closed set $E$ with empty intersection that cannot be separated in $\Psi(\mathcal{A})$ will be transformed into a pair of witnesses of non almost-normality in $\Psi(\mathcal{F})$. Now, let us obtain the required non almost-normal almost disjoint family of true cardinality $\mathfrak{c}$.\\

\noindent The following example is an instance of a machine for converting two almost disjoint families of the same cardinality, into a single almost disjoint family $\mathcal{A}$ with a countable set $\mathcal{C} \subset \mathcal{A}$ and a set $E \subset \Psi(\mathcal{A})$ such that $\mathcal{C}$ is closed and $E$ is regular closed in $\Psi(\mathcal{A})$, $\mathcal{C} \cap E = \emptyset$ and $\mathcal{A} \subset \mathcal{C} \cup E$.

\begin{example}
\label{TrueCountableC}
There is an almost disjoint family $\mathcal{A}$ of true cardinality $\mathfrak{c}$ on $\omega$ so that there is $\mathcal{C} \in [\mathcal{A}]^\omega$ and $W \in [\omega]^\omega$, such that $cl_{\Psi(\mathcal{A})}(W) \cap \mathcal{A} = \mathcal{A} \smallsetminus \mathcal{C}$. In particular, there is a non almost-normal almost disjoint family of true cardinality $\mathfrak{c}$.
\end{example}

\begin{proof}
Partition $\omega$ into two infinite disjoint sets $V,W$. Let $\mathcal{A}_0, \mathcal{A}_1$ be almost disjoint families of true cardinality $\mathfrak{c}$ on $V$ and $W$, respectively, and let $\mathcal{C} \in [\mathcal{A}_0]^\omega$. Now, a new family is built as follows, let $\alpha: \mathcal{A}_0 \smallsetminus \mathcal{C} \leftrightarrow \mathcal{A}_1$ be a bijective function. Let $\mathcal{A} = \{a \cup \alpha(a): a \in  \mathcal{A}_0 \smallsetminus \mathcal{C}\} \cup \mathcal{C}$.\\
Let us check that $\mathcal{A}$ is the desired family. Clearly, it is almost disjoint. To see that it has true cardinality $\mathfrak{c}$ let $M\subseteq \omega$ such that $|\mathcal{A}\upharpoonright _ M|\geq \omega$. Then, either $|\mathcal{C}\upharpoonright _ M|\geq \omega$ or  $|(\mathcal{A} \setminus \mathcal{C})\upharpoonright _ M|\geq \omega$. Hence, $|\mathcal{A}_0\upharpoonright _ M|\geq \omega$ or $|\mathcal{A}_1\upharpoonright _ M|\geq \omega$. Therefore, $|\mathcal{A}_0\upharpoonright _ M| = \mathfrak{c}$ or $|\mathcal{A}_1\upharpoonright _ M|= \mathfrak{c}$. In any case, $|\mathcal{A}\upharpoonright _ M|= \mathfrak{c}$. Thus, $\mathcal{A}$ is of true cardinality $\mathfrak{c}$.\\
Now, $a \in cl_{\Psi(\mathcal{A})}(W) \cap \mathcal{A} \leftrightarrow a \in \mathcal{A}$ $\wedge |a \cap W| = \omega \leftrightarrow a \in \mathcal{A} \wedge \big(\exists a_0 \in \mathcal{A}_0[a = a_0 \cup \alpha(a_0)]\big)$  $\leftrightarrow a \in \mathcal{A} \smallsetminus \mathcal{C}$.\
By Observation \ref{IfTrueCardcAandCtlbeComplementcannotseparate}, $\mathcal{A}$ is not almost-normal.
\end{proof}

\noindent If in the previous example we assume, in addition, that $\mathcal{A}_0, \mathcal{A}_1$ are mad families of the same cardinality, the resulting family $\mathcal{A}$ is mad as well: If $M \in [\omega]^\omega$, then $M$ has infinite intersection either with $V$ or with $W$, since $\mathcal{A}_0, \mathcal{A}_1$ are both mad, there is $a \in \mathcal{A}$ such $|a \cap M| = \omega$. Hence, the following holds:

\begin{corollary}
\label{TrueCountableCandmad}
The existence of a mad family of true cardinality $\mathfrak{c}$ implies the existence of a mad family $\mathcal{A}$ of true cardinality $\mathfrak{c}$ on $\omega$ so that there is $\mathcal{C} \in [\mathcal{A}]^\omega$ and $W \in [\omega]^\omega$, such that $cl_{\Psi(\mathcal{A})}(W) \cap \mathcal{A} = \mathcal{A} \smallsetminus \mathcal{C}$. In particular, the existence of a mad family of true cardinality $\mathfrak{c}$ implies the existence of a non almost-normal mad family of  true  cardinality $\mathfrak{c}$.
\end{corollary}

\begin{example}
\label{QuasiNormalNotAlmostNormalPsiSpace}
There is a quasi-normal not almost-normal almost disjoint family of true cardinality $\mathfrak{c}$.
\end{example}

\begin{proof}
Let $\mathcal{A}$ be a not almost-normal almost disjoint family of true cardinality $\mathfrak{c}$ as in Example \ref{TrueCountableC}. Hence, let $\mathcal{C} \in [\mathcal{A}]^\omega$ and $W \in [\omega]^\omega$, with $|\omega \smallsetminus W| = \omega$, such that $cl_{\Psi(\mathcal{A})}(W) \cap \mathcal{A} = \mathcal{A} \smallsetminus \mathcal{C}$. Consider the family of finite subsets of $[\omega]^\omega$, $\mathcal{E} = \big[[\omega]^\omega\big]^{<\omega}$ and let $\mathcal{B} = \{\{ C,D\} \in [\mathcal{E}]^2: (\bigcap C) \cap (\bigcap D) = \emptyset\}$. Since $|\mathcal{B}| = \mathfrak{c}$, we can list it as $\mathcal{B} = \{\{ C_\alpha,D_\alpha \} : \alpha < \mathfrak{c}\}$. A sequence of finite sets $\mathcal{F}_\alpha \in [\mathcal{A}]^{<\omega}$ will be built recursively in $\mathfrak{c}$ many steps.\\

\noindent For $\alpha = 0$, consider $\{ C_0,D_0 \} \in \mathcal{B}$. If for each $C \in C_0$ and $D \in D_0$, $\mathcal{A}\upharpoonright _C$ and $\mathcal{A}\upharpoonright _D$ all have size $\mathfrak{c}$, then for each $C \in C_0$ and $D \in D_0$ pick $a_C, b_D \in \mathcal{A} \setminus \mathcal{C}$ such that $|a_C\cap C| = \omega = |b_D\cap D|$ and all the $a_C$'s and $b_D$'s are distinct ($|\{a_C, b_D: C \in C_0, D \in D_0\}| = |C_0|+ |D_0|$). Let $\mathcal{F}_0 = \{a_C, b_D: C \in C_0, D \in D_0\}$. If there is $C \in C_0$ (or $D \in D_0$) such that $\mathcal{A}\upharpoonright _C$ is finite ($\mathcal{A}\upharpoonright _D$ is finite), let $\mathcal{F}_0 = \emptyset$. Observe that these are the only two possibilities as $\mathcal{A}$ is of true cardinality $\mathfrak{c}$.\\

\noindent Now assume $0 < \alpha < \mathfrak{c}$ and that for each $\beta < \alpha$, $\mathcal{F}_\beta$ is either empty of a finite subset of $ \mathcal{A} \setminus (\mathcal{C} \cup \bigcup _{\gamma < \beta}\mathcal{F}_\gamma)$. Consider the pair $\{ C_\alpha,D_\alpha \}$. If for each $C \in C_\alpha$ and $D \in D_\alpha$, $\mathcal{A}\upharpoonright _C$ and $\mathcal{A}\upharpoonright _D$ all have size $\mathfrak{c}$, then for each $C \in C_\alpha$ and $D \in D_\alpha$ pick $a_C, b_D \in \mathcal{A} \setminus  (\mathcal{C} \cup \bigcup _{\beta < \alpha}\mathcal{F}_\beta)$ such that $|a_C\cap C| = \omega = |b_D\cap D|$ and all the $a_C$'s and $b_D$'s are distinct ($|\{a_C, b_D: C \in C_\alpha, D \in D_\alpha\}| = |C_\alpha|+ |D_\alpha|$).  Let $\mathcal{F}_{\alpha} = \{a_C, b_D: C \in C_\alpha, D \in D_\alpha\}$. If there is $C \in C_\alpha$ (or $D \in D_\alpha$) such that $\mathcal{A}\upharpoonright _C$ is finite ($\mathcal{A}\upharpoonright _D$ is finite), let $\mathcal{F}_{\alpha} = \emptyset$. Let
$$\mathcal{F} = \big \{\bigcup \mathcal{F}_\alpha:\alpha < \mathfrak{c} \big \} \cup \big(\mathcal{A} \setminus \bigcup_{\alpha < \mathfrak{c}} \mathcal{F}_\alpha \big).$$

\noindent Since each $a \in \mathcal{F}$ is either an element of $\mathcal{A}$ or a finite union of elements of $\mathcal{A}$, it is clear that $\mathcal{F}$ is an almost disjoint family of true cardinality $\mathfrak{c}$.\\

\noindent {\bf Claim:} $\Psi(\mathcal{F})$ is quasi-normal.
\noindent Let $A \neq \emptyset \neq B$ be disjoint $\pi$-closed subsets of $\Psi(\mathcal{F})$. $A = \bigcap_{i=1}^nA_i$, $B = \bigcap_{j=1}^mB_j$, where each $A_i$ and $B_j$ are regular closed sets. It can be assumed that for each $i \leq n$ and for each $j \leq m$, $|A_i \cap \omega| = \omega = |B_j \cap \omega|$. Let $\alpha < \mathfrak{c}$ be minimal such that $C_\alpha = \{A_i \cap \omega:i\leq n\}$ and $D_\alpha = \{B_j \cap \omega:j\leq m\}$.\\
At stage $\alpha$, either $\mathcal{F}_{\alpha} = \emptyset$ or $\mathcal{F}_{\alpha} =  \{a_C, b_D: C \in C_\alpha, D \in D_\alpha\}$. The latter is not possible since for each $C \in C_\alpha$ and each $D \in D_\alpha$ the $a_C$'s and $b_D$'s  were chosen so that $|a_C \cap C| = \omega = |b_D \cap D|$ and this implies $\bigcup \mathcal{F}_{\alpha}$ is in the closure of each $C \in C_\alpha$ and each $D \in D_\alpha$ (see Observation \ref{closureofsubsetsofomega} and Observation \ref{regclosedset}). Hence $\bigcup \mathcal{F}_{\alpha} \in A \cap B$, but it is assumed that $A$ and $B$ are disjoint.\\

\noindent Thus,  $\mathcal{F}_{\alpha} = \emptyset$. This means that there exists $C \in C_\alpha$, such that $\mathcal{A}\upharpoonright _C = H$ for some finite set $H$ (or there exists $D \in D_\alpha$, such that $\mathcal{A}\upharpoonright _D = H$ for some finite set $H$). Without loss of generality assume there exists such $C \in C_\alpha$. Hence, $\mathcal{A} \upharpoonright _C = H_0$ for some finite set $H_0$. Observe that since for each $a \in \mathcal{F}$, either $a \in \mathcal{A}$ or $a$ is a finite union of elements of $\mathcal{A}$, then $\mathcal{F} \upharpoonright _C = H_1$ for some finite $H_1$ so that $|H_1| \leq |H_0|$. Now fix $i \leq n$ such that $A_i \cap \omega = C$. Since $A_i$ is regular closed, by \ref{regclosedset} $A_i \cap \mathcal{F} = H_0$. Thus, $A \cap \mathcal{F} \subseteq H_0$ and by Observation \ref{closedsets}, $A$ and $B$ can be separated. Therefore $\Psi(\mathcal{F})$ is quasi-normal.\\

\noindent {\bf Claim:} $\Psi(\mathcal{F})$ is not almost-normal.\\
Fix $a \in \mathcal{F} \smallsetminus \mathcal{C}$, then $a \in \mathcal{A} \smallsetminus \mathcal{C}$ or $a$ is a finite union of elements of $\mathcal{A} \smallsetminus \mathcal{C}$. Since $cl_{\Psi(\mathcal{A})}(W) \cap \mathcal{A} = \mathcal{A} \smallsetminus \mathcal{C}$, $|W \cap a| = \omega$. Hence, $a \in  cl_{\Psi(\mathcal{F})}(W)$, i.e., $\mathcal{F} \smallsetminus \mathcal{C} \subseteq cl_{\Psi(\mathcal{F})}(W)$. On the other hand, if $c \in \mathcal{C}$, $c \notin  cl_{\Psi(\mathcal{A})}(W)$, thus $|c \cap W| < \omega$ and therefore $c \notin  cl_{\Psi(\mathcal{F})}(W)$. \\
Hence, $\mathcal{C}$ is a closed set, $cl_{\Psi(\mathcal{F})}(W)$ is a regular closed set, they do not intersect and by Observation \ref{IfTrueCardcAandCtlbeComplementcannotseparate} they cannot be separated.
\end{proof}

\noindent If in the construction of Example \ref{QuasiNormalNotAlmostNormalPsiSpace}, a mad family as in Corollary \ref{TrueCountableCandmad} is chosen, then the resulting family $\mathcal{F}$ is mad, quasi-normal and not almost-normal. Thus:

\begin{corollary}
\label{TruemadQuasiNotAlmNor}
The existence of a mad family of true cardinality $\mathfrak{c}$ implies the existence of a quasi-normal, non almost-normal mad family of true cardinality $\mathfrak{c}$.
\end{corollary}

\noindent The following example provides a mildly-normal not partly-normal almost disjoint family $\mathcal{F}$ of true cardinality $\mathfrak{c}$ which is constructed using three almost disjoint families of true cardinality $\mathfrak{c}$. In order to make $\mathcal{F}$ mildly-normal all pairs of disjoint regular closed sets in $\Psi(\mathcal{F})$ have to be separated. A similar approach as in Example \ref{QuasiNormalNotAlmostNormalPsiSpace} is followed. It will be  possible to build $\mathcal{F}$ so that all pairs of disjoint regular closed sets in $\Psi(\mathcal{F})$ will be trivial, i.e., one of them will have finite intersection with $\mathcal{F}$ (Observation \ref{closedsets} guarantees they can be separated). To make $\mathcal{F}$ not quasi-normal, there will be a regular closed set $A$ disjoint from a $\pi$-closed set $B$ that cannot be separated. The basic idea is to partition $\omega$ into three infinite sets, $W$, $V_0$, $V_1$, take an almost disjoint family of true cardinality $\mathfrak{c}$ on each one of them (we use the property of true cardinality $\mathfrak{c}$ to make $\mathcal{F}$ mildly-normal), and build $\mathcal{F}$ so that in $\Psi(\mathcal{F})$,  $A = cl_{\Psi(\mathcal{F})}(W)$ and 
$B = cl_{\Psi(\mathcal{F})}(V_0) \cap cl_{\Psi(\mathcal{F})}(V_1)$ are disjoint but cannot be separated.

\begin{example}
\label{mildlynotpartlyNormal}
There exists a mildly-normal not partly-normal almost disjoint family of true cardinality $\mathfrak{c}$.
\end{example}

\begin{proof}
Partition $\omega$ into three disjoint infinite pieces, that is $W, V_0, V_1 \in [\omega]^{\omega}$ and $W \cup V_0 \cup V_1 = \omega$.  If $Y \in \{W, V_0, V_1\}$ let $\mathcal{A}_Y$ be an almost disjoint family of true cardinality $\mathfrak{c}$ on $Y$. List all pairs of infinite subsets of $\omega$ with empty intersection as $\{ \{ C_\alpha , D_\alpha \}: \alpha < \mathfrak{c} \}$. A sequence of finite sets $\mathcal{F}_\alpha \subset \mathcal{A}_{W} \cup \mathcal{A}_{V_0} \cup \mathcal{A}_{V_1}$ will be built recursively in $\mathfrak{c}$ many steps.\\

\noindent Fix $\alpha < \mathfrak{c}$, assume that for each $\beta < \alpha$, $\mathcal{F}_\beta$ has been defined such that $\mathcal{F}_\beta$ is a possibly empty finite set $\mathcal{F}_\beta \subset (\mathcal{A}_{W} \cup \mathcal{A}_{V_0} \cup \mathcal{A}_{V_1}) \setminus \bigcup _{\gamma < \beta}\mathcal{F}_\gamma$ such that either $\mathcal{F}_\beta \subset \mathcal{A}_{W}$ or $\mathcal{F}_\beta$ has nonempty intersection with exactly two elements of $\{\mathcal{A}_{W}, \mathcal{A}_{V_0}, \mathcal{A}_{V_1}\}$. Consider $\{C_\alpha, D_\alpha\}$.\\

\noindent {\bf Case 1:} Either all three sets $\mathcal{A}_{W}\upharpoonright _{C_\alpha}$, $\mathcal{A}_{V_0}\upharpoonright _{C_\alpha}$, $\mathcal{A}_{V_1}\upharpoonright _{C_\alpha}$ are finite, or all three sets $\mathcal{A}_{W}\upharpoonright _{D_\alpha}$, $\mathcal{A}_{V_0}\upharpoonright _{D_\alpha}$, $\mathcal{A}_{V_1}\upharpoonright _{D_\alpha}$ are finite. In this case, let $\mathcal{F}_{\alpha} = \emptyset$.\\

\noindent {\bf Case 2:} Case 1 is false. That is (given that $\mathcal{A}_{W}$, $\mathcal{A}_{V_0}$, $\mathcal{A}_{V_1}$ are of true cardinality $\mathfrak{c}$): at least one of the three sets $\mathcal{A}_{W}\upharpoonright _{C_\alpha}$, $\mathcal{A}_{V_0}\upharpoonright _{C_\alpha}$, $\mathcal{A}_{V_1}\upharpoonright _{C_\alpha}$ has size $\mathfrak{c}$ and at least one of the three sets $\mathcal{A}_{W}\upharpoonright _{D_\alpha}$, $\mathcal{A}_{V_0}\upharpoonright _{D_\alpha}$, $\mathcal{A}_{V_1}\upharpoonright _{D_\alpha}$ has size $\mathfrak{c}$. Choose the smallest $i$ such that Subcase $2.i$ (below) holds, define $\mathcal{F}_{\alpha}$ accordingly, and ignore the other subcases.\\
\noindent {\bf Subcase 2.1:} $|\mathcal{A}_{W}\upharpoonright _{C_\alpha}| = \mathfrak{c}=|\mathcal{A}_{W}\upharpoonright _{D_\alpha}|$. Pick $c_\alpha, d_\alpha \in \mathcal{A}_{W}\setminus \bigcup _{\beta < \alpha}\mathcal{F}_\beta$ such that $c_\alpha \neq d_\alpha$ and $|c_\alpha \cap C_\alpha| = \omega = |d_\alpha \cap D_\alpha|$. Let $\mathcal{F}_{\alpha} = \{c_\alpha, d_\alpha\}$.\\
\noindent {\bf Subcase 2.2:} There exists $i \in \{0,1\}$ so that $|\mathcal{A}_{V_i}\upharpoonright _{C_\alpha}| = \mathfrak{c} = |\mathcal{A}_{V_i}\upharpoonright _{D_\alpha}|$. Pick $c_\alpha, d_\alpha \in \mathcal{A}_{V_i}\setminus \bigcup _{\beta < \alpha}\mathcal{F}_\beta$, such that $c_\alpha \neq d_\alpha$ and $|c_\alpha \cap C_\alpha| = \omega = |d_\alpha \cap D_\alpha|$. In addition, pick $e_\alpha \in \mathcal{A}_{V_{1-i}}\setminus\bigcup _{\beta < \alpha}\mathcal{F}_\beta$. Let $\mathcal{F}_{\alpha}  = \{c_\alpha, d_\alpha, e_\alpha\}$.\\
\noindent {\bf Subcase 2.3:}  $|\mathcal{A}_{V_0}\upharpoonright _{C_\alpha}| = \mathfrak{c}=|\mathcal{A}_{V_1}\upharpoonright _{D_\alpha}|$. Pick $c_\alpha \in \mathcal{A}_{V_0}\setminus \bigcup _{\beta < \alpha}\mathcal{F}_\beta$ and $d_\alpha \in \mathcal{A}_{V_1}\setminus \bigcup _{\beta < \alpha}\mathcal{F}_\beta$ such that $|c_\alpha \cap C_\alpha| = \omega = |d_\alpha \cap D_\alpha|$  and let $\mathcal{F}_{\alpha} = \{c_\alpha, d_\alpha\}$.\\
\noindent {\bf Subcase 2.4:} $|\mathcal{A}_W\upharpoonright _{C_\alpha}| = \mathfrak{c}$ and there exists $i \in \{0,1\}$ so that $|\mathcal{A}_{V_i}\upharpoonright _{D_\alpha}|= \mathfrak{c}$. Pick $c_\alpha \in \mathcal{A}_W\setminus \bigcup _{\beta < \alpha}\mathcal{F}_\beta$ and $d_\alpha \in \mathcal{A}_{V_i}\setminus \bigcup _{\beta < \alpha}\mathcal{F}_\beta$ such that $|c_\alpha \cap C_\alpha| = \omega = |d_\alpha \cap D_\alpha|$ and let $\mathcal{F}_{\alpha} = \{c_\alpha, d_\alpha\}$.\\

\noindent This finishes Case 2 and the construction of $\mathcal{F}_\alpha$ for $\alpha < \mathfrak{c}$. Let
$$\mathcal{F} = \big \{\bigcup \mathcal{F}_\alpha:\alpha < \mathfrak{c} \big \} \cup \big((\mathcal{A}_{W} \cup \mathcal{A}_{V_0} \cup \mathcal{A}_{V_1}) \setminus \bigcup_{\alpha < \mathfrak{c}} \mathcal{F}_\alpha \big).$$

\noindent It will be shown that $\mathcal{F}$ is the desired almost disjoint family. Given that each of $\mathcal{A}_{W}$, $\mathcal{A}_{V_0}$ and $\mathcal{A}_{V_1}$ is of true cardinality $\mathfrak{c}$ and if we let $a \in \mathcal{F}$, then either $a$ is an element or a finite union of elements of $\mathcal{A}_{W} \cup \mathcal{A}_{V_0} \cup \mathcal{A}_{V_1}$, then $\mathcal{F}$ is an almost disjoint family of true cardinality $\mathfrak{c}$.\\

\noindent {\bf $\Psi(\mathcal{F})$ is not partly-normal:}\\
\noindent Let $A = cl_{\Psi(\mathcal{F})}(W)$ and 
$B = cl_{\Psi(\mathcal{F})}(V_0) \cap cl_{\Psi(\mathcal{F})}(V_1)$.
By Observation \ref{closureofsubsetsofomega}, $A$ is regular closed and $B$ is a $\pi$-closed set. Observe that since $\mathcal{A}_{V_0}$ and  $\mathcal{A}_{V_1}$ are of true cardinality $\mathfrak{c}$, there are infinite many pairs $\{C_\alpha,D_\alpha\}$ such that $C_\alpha \subset V_0$, $D_\alpha \subset V_1$, and $|\mathcal{A}_{V_0} \upharpoonright _{C_\alpha}| = \mathfrak{c} = |\mathcal{A}_{V_1} \upharpoonright _{D_\alpha}|$. For such pairs Subcase 2.3 applies and therefore $|B \cap \mathcal{F}| \ge \omega$. In addition, $A \cap B = \emptyset$: assume there is $a \in A \cap B$. Since $V_0 \cap V_1 = \emptyset$, $B \cap \omega = \emptyset$, hence $a \in \mathcal{F} \cap A \cap B$. By Observation \ref{regclosedset}, $|a \cap W|= |a \cap V_0|= |a \cap V_1| = \omega$. This implies that $a \notin \mathcal{A}_W \cup \mathcal{A}_{V_0}  \cup\mathcal{A}_{V_1}$. There is $\alpha < \mathfrak{c}$ such that $a = \bigcup \mathcal{F}_{\alpha}$, but by the construction, $\mathcal{F}_{\alpha} \subset  \mathcal{A}_W$ or $\mathcal{F}_{\alpha}$ intersects exactly two elements of $\{\mathcal{A}_W, \mathcal{A}_{V_0},\mathcal{A}_{V_1}\}$ which contradicts that $a$ has infinite intersection with $W$, $V_0$ and $V_1$. Whence, $A\cap B = \emptyset$.\\
\noindent It remains to show that $A$ and $B$ cannot be separated. Assume, on the contrary, that there are $S, T \subseteq \Psi(\mathcal{F})$ open such that $A \subseteq S$, $B \subseteq T$ and $S \cap T = \emptyset$. Let $\alpha < \mathfrak{c}$ such that $C_\alpha = \omega \cap S$ and $D_\alpha = \omega \cap T$. For the pair $\{C_\alpha,D_\alpha\}$, either Case 1 or Case 2 of the construction holds.\\

\noindent {\it If Case 1 holds:} since $W \subseteq C_\alpha$, $\mathcal{A}_{W}\upharpoonright _{C_\alpha}$ is not finite. Hence, $\mathcal{A}_{W}\upharpoonright _{D_\alpha}$, $\mathcal{A}_{V_0}\upharpoonright _{D_\alpha}$, $\mathcal{A}_{V_1}\upharpoonright _{D_\alpha}$ are finite. Thus, $\mathcal{F}\upharpoonright _{D_\alpha}$ is finite. Since $cl_{\Psi(\mathcal{F})}(D_\alpha)$ is regular closed and $\mathcal{F}\upharpoonright _{D_\alpha}$ is finite, by Observation \ref{regclosedset}, $\mathcal{F} \cap cl_{\Psi(\mathcal{F})}(D_\alpha)$ is finite. Now, $T$ is open and $D_\alpha = \omega \cap T$, therefore $T \subseteq cl_{\Psi(\mathcal{F})}(D_\alpha)$. Hence, $\mathcal{F} \cap T$ is finite. Given that $|B \cap \mathcal{F}| \ge \omega$, $B \not \subseteq T$, which is a contradiction.\\
\noindent {\it If Case 2 holds:} Either $\mathcal{F}_{\alpha} \subset \mathcal{A}_W$ or $\mathcal{F}_{\alpha}$ intersects exactly two elements of $\{\mathcal{A}_W, \mathcal{A}_{V_0},\mathcal{A}_{V_1}\}$. In any case $\bigcup \mathcal{F}_{\alpha}$ is an element of $A$ or $B$. In addition, there exist $c_\alpha, d_\alpha \in \mathcal{F}_{\alpha}$ such that $|c_\alpha \cap C_\alpha| = \omega = |d_\alpha \cap D_\alpha|$. If $\bigcup \mathcal{F}_{\alpha} \in A$, then for each open neighbourhood $U$ of $\bigcup \mathcal{F}_{\alpha}$, $U \cap T \neq \emptyset$ (which implies $U \not \subseteq S$), and this contradicts that $S$ is open. We reach a similar contradiction if $\bigcup \mathcal{F}_{\alpha} \in B$. Hence, $A$ and $B$ cannot be separated.\\

\noindent {\bf $\Psi(\mathcal{F})$ is mildly-normal:}\\
\noindent Let $C \neq \emptyset \neq D$ be disjoint regular closed subsets of $\Psi(\mathcal{F})$. It can be assumed that $|C \cap \omega| = \omega = |D \cap \omega|$. Fix $\alpha < \mathfrak{c}$ such that $C \cap \omega = C_\alpha$ and $D \cap \omega = D_\alpha$. For the pair $\{C_\alpha,D_\alpha\}$, either Case 1 or Case 2 holds. If Case 2 holds, there exist $c_\alpha, d_\alpha \in \mathcal{F}_{\alpha}$ such that $|c_\alpha \cap C_\alpha| = \omega = |d_\alpha \cap D_\alpha|$. Thus, $\bigcup \mathcal{F}_{\alpha} \in cl_{\Psi(\mathcal{F})}(C_\alpha) \cap cl_{\Psi(\mathcal{F})}(D_\alpha)$ $\subseteq cl_{\Psi(\mathcal{F})}(C) \cap cl_{\Psi(\mathcal{F})}(D) = C \cap D$. This contradicts $C \cap D = \emptyset$.\\
Thus, Case 1 holds. This means that all three sets $\mathcal{A}_{W}\upharpoonright _{C_\alpha}$, $\mathcal{A}_{V_0}\upharpoonright _{C_\alpha}$, $\mathcal{A}_{V_1}\upharpoonright _{C_\alpha}$ are finite, or all three sets $\mathcal{A}_{W}\upharpoonright _{D_\alpha}$, $\mathcal{A}_{V_0}\upharpoonright _{D_\alpha}$, $\mathcal{A}_{V_1}\upharpoonright _{D_\alpha}$ are finite.\\
Without loss of generality, assume the former. This implies that $\mathcal{F}\upharpoonright _{C_\alpha}$ is finite. Given that $C$ is a regular closed set and $C_\alpha = C\cap \omega$, by Observation \ref{regclosedset} $C \cap \mathcal{F}$ is finite and by Observation \ref{closedsets}, $C$ and $D$ can be separated. Therefore $\Psi(\mathcal{F})$ is mildly-normal.
\end{proof}

\noindent Observe that if in the construction of Example \ref{mildlynotpartlyNormal}, the families  $\mathcal{A}_{W}$, $\mathcal{A}_{V_0}$ and $\mathcal{A}_{V_1}$ are mad of true cardinality $\mathfrak{c}$, then the family $\mathcal{F}$ is mad as well. Therefore:

\begin{corollary}
\label{TruemadMildlyNotPartly}
If there exists a mad family of true cardinality $\mathfrak{c}$, then there is a mildly-normal, not partly-normal mad family of true cardinality $\mathfrak{c}$.
\end{corollary}

\begin{definition}
\label{nPartlyNormal}
For a positive $n\in \omega$, a regular space will be called \emph{$n$-partly-normal} if any two nonintersecting sets $A$ and $B$, where $A$ is regularly closed and $B$ is the intersection of at most $n$ regularly closed sets, are separated.
\end{definition}

\noindent Observe that $1$-partly-normal coincides with mildly-normal, and for each positive $n\in \omega$, partly-normal $\rightarrow$ $(n+1)$-partly-normal $\rightarrow$ $n$-partly-normal $\rightarrow$ mildly-normal. It is possible to extend the idea in Example \ref{mildlynotpartlyNormal} (partition $\omega$ into $n+2$ pairwise disjoint infinite pieces, take an almost disjoint family of true cardinality $\mathfrak{c}$ on each piece and let $\{\mathbb{C}_\alpha:\alpha < \mathfrak{c}\}$ list all sets $\mathbb{C} \subset [\omega]^\omega$ such that $2 \le |\mathbb{C}| \le n+1$), to show the following:

\begin{theorem}
\label{npartlyNormalnotn+1}
For each positive $n\in \omega$, there exists a $n$-partly-normal not $(n+1)$-partly-normal almost disjoint family of true cardinality $\mathfrak{c}$.
\end{theorem}

\noindent Similarly as Corollary \ref{TruemadMildlyNotPartly}, it also holds true:

\begin{corollary}
\label{TruemadNPartlyNotn+1}
If there exists a mad family of true cardinality $\mathfrak{c}$, then for each positive $n\in \omega$, there is a $n$-partly-normal not $(n+1)$-partly-normal mad family of true cardinality $\mathfrak{c}$.
\end{corollary}

\noindent Corollary \ref{TruemadQuasiNotAlmNor} says, in particular, that there is a quasi-normal mad family, provided there is a completely separable mad family. Our next example shows that, assuming {\bf CH}, not only a quasi-normal mad family exists, but one that it is also Luzin. Recall that an almost disjoint family is \emph{Luzin} if it can be enumerated as $\{a_\alpha:\alpha <\omega_1\}$ so that for each $\alpha < \omega_1$ and each $n \in \omega$, $\{\beta < \alpha : a_\alpha \cap a_\beta \subseteq n \}$ is finite. Luzin introduced this kind almost disjoint family in \cite{Lu} to provide an example of an almost disjoint family $\mathcal{A}$ such that every pair of uncountable subfamilies of $\mathcal{A}$ have no \emph{separation} (it will be said that two subfamiles $\mathcal{B}$ and $\mathcal{C}$ of $\mathcal{A}$, have a separation if there is $X \subseteq \omega$  such that for each $b \in \mathcal{B}$, $b \subseteq ^* X$ and for each $c \in \mathcal{C}$, $c \cap X =^* \emptyset$). Thus, Luzin families are far from being normal. No mad family is normal, no Luzin family is normal, and yet, there is, consistently, a quasi-normal Luzin mad family. 

\begin{example}[CH]
\label{MadLuzinQuasi}
There is a Luzin mad family $\mathcal{A}$ which is quasi-normal.
\end{example}

\begin{proof}
The standard construction of a Luzin family is modified to build a family $\mathcal{A}$ with the extra following property: for each $X \subseteq \omega$, either $X$ is covered by finitely many elements of $\mathcal{A}$ or the set of elements of $\mathcal{A}$ that has finite intersection with $X$ is countable.\\
The idea is to use {\bf CH} to list all infinite subsets $X_\alpha \subseteq \omega$, with $\alpha < \omega_1$ and, at stage $\alpha < \omega_1$ of the construction of the family, $X_\alpha$ will be covered by the $\alpha$-th element of the family, together with finitely many elements of the family previously constructed or, if $X_\alpha$ has infinite intersection with infinitely many elements of the family constructed so far, it will be guaranteed that, from that stage until the end, all elements of the family will have infinite intersection with $X_\alpha$.\\
Partition $\omega$ into infinite pairwise disjoint subsets $a_i$, with $i \in \omega$, that is $\omega = \bigcup _{i\in\omega}a_i$, and $i \neq j $ implies $a_i \cap a_j = \emptyset$. List all infinite subsets of $\omega$ as $[\omega]^\omega = \{X_\alpha:\alpha < \omega_1\}$ such that for each $n \in \omega$, $X_n = a_n$. If $\alpha$ is such that $\omega \leq \alpha < \omega_1$, recursively assume we have constructed $a_\beta$ for $\beta < \alpha$ such that $\{a_\beta: \beta < \alpha\}$ is an almost disjoint family and for each  $\beta < \alpha$, $X_\beta$ is covered by finitely many elements of $\{a_\gamma: \gamma \le \beta\}$ or for each $\beta \le \gamma < \alpha$, $|X_\beta \cap a_\gamma| = \omega$.\\
The $\alpha$-th element of the family will be constructed. Reenumerate the sets $\mathcal{A}_\alpha = \{a_\beta : \beta < \alpha\}$ and $J_\alpha = \{X_\beta : \beta \leq \alpha\}$ as $\mathcal{A}_\alpha = \{a^\alpha_n : n \in \omega\}$ and $J_\alpha = \{X^\alpha_n : n \in \omega\}$. Let $I_\alpha = \{n \in \omega: X_n^\alpha \in \mathcal{I}^+(\mathcal{A_\alpha})\}$.\\
There are two cases, either $X_\alpha \in \mathcal{I}^+(\mathcal{A_\alpha})$ or $X_\alpha \notin \mathcal{I}^+(\mathcal{A_\alpha})$. We will construct $a_\alpha$ depending on whether at this stage, $I_\alpha$ is still empty or not.\\
If $I_\alpha = \emptyset$ (observe that in particular $X_\alpha \notin \mathcal{I}^+(\mathcal{A_\alpha})$), let $p_n^\alpha \subseteq a_n^\alpha \setminus \bigcup _{i < n}a_i^\alpha$, such that $|p_n^\alpha| = n$ and let $a_\alpha = \bigcup _{n \in \omega} p_n^\alpha \cup (X_\alpha \setminus\bigcup (\mathcal{A}_\alpha \upharpoonright _{X_\alpha}))$.\\
If $I_\alpha \neq \emptyset$. Let $\{Y_n:n\in \omega\}$ list all $X_n^\alpha$ such that $n \in I_\alpha$ and so that not only each $X_n^\alpha$ appears infinitely often but for each $n \in I_\alpha$ and for each $m \in \omega$, there is some $s \geq m$ such that $Y_s = X_n^\alpha$ and $|a_s^\alpha \cap Y_s| = \omega$. For $n \in \omega$, if $|a_n^\alpha \cap Y_n| < \omega$, let $p_n^\alpha \subseteq a_n^\alpha \setminus \bigcup _{i < n}a_i^\alpha$, such that $|p_n^\alpha| = n$. If $|a_n^\alpha \cap Y_n| = \omega$, let $p_n^\alpha \subseteq (a_n^\alpha \setminus \bigcup _{i < n}a_i^\alpha) \cap Y_n$ such that $|p_n^\alpha|=n$. Let $a_\alpha = \bigcup_{n \in \omega}p_n^\alpha \cup (X_\alpha \setminus\bigcup (\mathcal{A}_\alpha \upharpoonright _{X_\alpha}))$. Observe that  if $X_\alpha \notin \mathcal{I}^+(\mathcal{A_\alpha})$, then the construction of $a_\alpha$ guarantees that $X_\alpha$ is covered by finitely many elements of $\mathcal{A}_\alpha  \cup \{a_\alpha\}$. On the other hand, if $X_\alpha \in \mathcal{I}^+(\mathcal{A_\alpha})$, then $X_\alpha$ appears infinitely often in $\{Y_n:n\in \omega\}$, thus, it has infinite intersection with $a_\alpha$ and it will have infinite intersection with each $a_\beta$ for each $\beta > \alpha$.

\noindent Finally, let $\mathcal{A} = \{a_\alpha:\alpha < \omega_1\}$. The construction guarantees that $\mathcal{A}$ is Luzin: let $\alpha \in \omega_1$ and $n \in \omega$. Recall that $\mathcal{A}_\alpha = \{a_\beta : \beta < \alpha\} = \{a_m^\alpha: m \in \omega\}$ and for each $m \ge n$, $p_m^\alpha \subseteq a_\alpha \cap a_m^\alpha$ and $|p_m^\alpha| = m \ge n$. Hence, $\{\beta < \alpha: a_\beta\cap a_\alpha \subseteq n\}$ is finite. Let us verify that it is mad. Let $\alpha, \beta \in \omega_1$ such that $\beta < \alpha$. There is $n \in \omega$ with $a_\beta = a_n^\alpha$. Observe that for $i\leq n$, $p_i ^\alpha$ is finite, for $i >n$, $p_i^\alpha \cap a_\beta = \emptyset$ and, $ (X_\alpha \setminus\bigcup (\mathcal{A}_\alpha \upharpoonright _{X_\alpha})) \cap a_\beta$ is finite. Hence, $a_\beta \cap a_\alpha$ is finite. Now, let $X \in [\omega]^\omega$ and $\alpha < \omega_1$ such that $X = X_\alpha$. Either $X \notin \mathcal{I}^+(\mathcal{A}_\alpha)$, in which case $X$ is covered by finitely many elements of $\mathcal{A}_\alpha \cup \{a_\alpha\}$ (i.e. $X$ has infinite intersection with some element of $\mathcal{A}$), or $X \in \mathcal{I}^+(\mathcal{A}_\alpha)$, in which case for each $\gamma > \alpha$, $|X \cap a_\gamma| = \omega$. Thus, $\mathcal{A}$ is mad and it has the desired property. \\

\noindent Let us show that $\mathcal{A}$ is quasi-normal. Let $A, B \subseteq \Psi(\mathcal{A})$ such that $A$ and $B$ are $\pi$-closed sets and $A \cap B = \emptyset$. Thus, $A = \bigcap _{i<n}A_i$, $B = \bigcap _{j<m}B_j$, where each $A_i$, $B_j$ are regular closed subsets of $\Psi(\mathcal{A})$ for $i<n$ and $j<m$. Assume that for each $i<n$ and for each $j<m$, $|A_i \cap \mathcal{A}| \ge \omega$ and $|B_j \cap \mathcal{A}| \ge \omega$. Hence, for each $i<n$ and for each $j<m$, $A_i \cap \omega \in \mathcal{I}^+(\mathcal{A})$ and $B_j \cap \omega \in \mathcal{I}^+(\mathcal{A})$. By the construction of $\mathcal{A}$, for each $i<n$ and for each $j<m$ the sets $\{a \in \mathcal{A}:|a \cap (A_i \cap \omega)|< \omega \}$ and $\{a \in \mathcal{A}:|a \cap (B_j \cap \omega)|< \omega \}$ are countable. Since the $A_i$'s, $B_j$'s are regular closed sets, then for each $i<n$ and for each $j<m$, $|\mathcal{A} \setminus A_i| \leq \omega$ and $|\mathcal{A} \setminus B_j | \leq \omega$. Thus, $|\mathcal{A} \setminus \bigcap_{i<n}A_i| \leq \omega$ and $|\mathcal{A} \setminus \bigcap_{j<m} B_j | \leq \omega$. Therefore, $A \cap B \neq \emptyset$. Hence, there exists some $i<n$ (or $j<m$), such that $|A_i \cap \mathcal{A}| < \omega$ ($|B_j \cap \mathcal{A}| < \omega$). Then $|A \cap \mathcal{A}| < \omega$ ($|B \cap \mathcal{A}| < \omega$) and, by Observation \ref{closedsets}, $A$ can be separated from $B$.
\end{proof}

\section{Strongly $\aleph_0$-separated almost disjoint families}

It is still open whether there could be (e.g., assuming CH) a mad family whose $\Psi$-space is almost normal, or one whose $\Psi$-space is almost normal but not normal.  However, we can construct a mad family with a slightly weaker property: 

\begin{definition}
\label{stronglycountableseparated}
An almost disjoint family $\mathcal{A}$ will be called \emph{strongly $\aleph_0$-separated}, if and only if for each pair of disjoint countable subfamilies there is a clopen partition of $\mathcal{A}$ that separates them. That is, for each $A,B \in [\mathcal{A}]^\omega$, with  $A\cap B = \emptyset$, there is $X \subset \omega$ such that 
\begin{enumerate}[nosep]
\item For each $a \in \mathcal{A}$, $a \subseteq ^* X$ or $a \cap  X =^* \emptyset$,
\item For each $a \in A$, $a \subseteq ^* X$,
\item For each $a \in B$, $a \cap  X =^* \emptyset$.
\end{enumerate}
\end{definition}

\begin{lemma}
\label{quasiseparation}
Almost-normal almost disjoint families are strongly $\aleph_0$-separated.
\end{lemma}

\begin{proof}
Let $\mathcal{A}$ be an almost-normal almost disjoint family. First, let us recall that each pair of disjoint countable closed subsets of a regular space can be separated. Hence, given that $\Psi(\mathcal{A})$ is regular and $\mathcal{A}$ is a closed discrete subset of $\Psi(\mathcal{A})$, if we consider $A, B \in [\mathcal{A}]^\omega$ so that $A \cap B = \emptyset$, then $A$ and $B$  can be separated. Thus, there exist $U_A, U_B$ open subsets of  $\mathcal{A}$ such that $U_A \cap U_B = \emptyset$ and $A \subseteq U_A$, $B \subseteq U_B$. Let $C = cl_{\Psi(\mathcal{A})}(U_A \cap \omega)$. By Observation \ref{closureofsubsetsofomega}, $C$ is a regular closed set. Then $C$ and $\mathcal{A} \setminus C$ is a pair of a regular closed set and a closed set with empty intersection. Since $\mathcal{A}$ is almost-normal, there exist $V$, $W$ open subsets of $\Psi(\mathcal{A})$ such that $V \cap W = \emptyset$ and $C \subseteq V$, $\mathcal{A} \setminus C \subseteq W$.\\
Let us check that $X = V \cap \omega$ has the desired properties. Indeed, let $a \in \mathcal{A}$, if $a \in C$, then $a \subseteq ^* V \cap \omega = X$. If $a \in \mathcal{A}\setminus C$, then $a\subseteq ^* W \cap \omega$, thus $a \cap X =^* \emptyset$. Now, if $a \in A$, $a \subseteq ^* U_A \cap \omega \subseteq C \cap \omega \subseteq V \cap \omega = X$. If $b \in B$, $|b \cap U_A| < \omega$ thus, $b \in \mathcal{A} \setminus C$. Hence, $b \subseteq ^* W$, i.e. $b \cap X =^* \emptyset$. Hence, $\mathcal{A}$ is strongly $\aleph_0$-separated.
\end{proof}

\begin{proposition}[CH]
\label{CHimpliesMADstronglycountableseparated}
There is a strongly $\aleph_0$-separated mad family.
\end{proposition}

\begin{proof}
Let $\{(A_\beta,B_\beta)\in [\omega_1]^\omega \times [\omega_1]^\omega :  \omega \le \beta < \omega_1\}$ list all disjoint pairs of countable subsets of $\omega_1$ in such a way that for each $\omega \le \beta < \omega_1$, $A_\beta \cup B_\beta \subseteq \beta$. In addition, list $[\omega]^\omega$ as $\{Y_\alpha:  \omega \le \alpha < \omega_1\}$.\\
Let $\omega \le \alpha <\omega_1$ and assume that for each $\omega \le \beta < \alpha$, the sets $X_\beta$, $a_\beta \subset \omega$ have been defined such that:

\begin{enumerate}[nosep]
\item For each $\gamma \in A_\beta: a_\gamma \subseteq^* X_\beta$,
\item For each $\gamma \in B_\beta: a_\gamma \cap  X_\beta =^* \emptyset$,
\item For each $\gamma < \alpha: a_\gamma \subseteq^* X_\beta$ or $a_\gamma \cap  X_\beta =^* \emptyset$,
\item If there is $\gamma < \beta$ such that $|a_\gamma \cap Y_\beta| = \omega$, then $a_\beta = \emptyset$. Otherwise, $|a_\beta \cap Y_\beta|=\omega$,
\item For each $\eta, \gamma < \alpha$, $a_\eta \cap a_\gamma =^* \emptyset$,
\end{enumerate} 

\noindent Let us construct $X_\alpha$. List $\alpha \smallsetminus B_\alpha$ and $B_\alpha$ as $\alpha \smallsetminus B_\alpha =\{\gamma _n:n\in \omega\}$, $B_\alpha=\{\beta _n:n\in \omega\}$. Since $A_\alpha \cup B_\alpha \subseteq \alpha$, then $A_\alpha \subseteq \alpha \smallsetminus B_\alpha$ and for each $n \in \omega$, $\gamma_n, \beta_n < \alpha$. That is, $a_{\gamma_n},a_{\beta_n}$ have been defined. In addition, for $n \in \omega$, $W_n =a_{\gamma_n}\smallsetminus [\bigcup_{j\le n}a_{\beta_j}]$ is either empty of infinite. Define $X_\alpha = \bigcup_{n\in \omega}W_n$. Observe that $(A_\alpha,B_\alpha)$ and $X_\alpha$ satisfy properties 1. and 2. of the recursive construction.\\
Now let us build $a_\alpha$. Reenumerate $\{X_\beta: \beta <\alpha\}$$\cup \{X_\alpha\}$ as $\{X^n:n\in \omega\}$.  For $n \in \omega$, let $X_1^n = X^n$, $X_0^n = \omega \smallsetminus X^n$. 
If there is $\gamma < \alpha$ such that $|a_\gamma \cap Y_\alpha| = \omega$, then  let $a_\alpha = \emptyset$.
On the other hand, if for each $\gamma < \alpha$, $|a_\gamma \cap Y_\alpha| < \omega$, for $n \in \omega$, pick $i(n) \in \{0,1\}$ so that $Y_\alpha \cap \bigcap_{j \le n}X_{i(j)}^j$ is infinite. For each $n \in \omega$, pick $p_n \in \big[Y_\alpha \cap \bigcap_{j \le n}X_{i(j)}^j\big] \smallsetminus \{p_j:j<n\}$. In this case, let $a_\alpha = \{p_n:n\in\omega\}$. Since $a_\alpha \subseteq Y_\alpha$, then for ech $\beta < \alpha$, $a_\beta \cap a_\alpha$ is finite.\\
This finishes the recursive construction of $X_\alpha$ and $a_\alpha$. Regardless of whether $a_\alpha$ is empty or not, it satisfies properties 4. and 5. In addition, it holds true that for each $\gamma, \beta \le \alpha$: $a_\gamma \subseteq ^* X_\beta$ or $a_\gamma \cap X_\beta =^* \emptyset$. Thus, property 3. is satisfied.\\
Let $\mathcal{A} = \{a_\alpha: \omega \le \alpha < \omega_1 \textit{ and } a_\alpha \neq \emptyset \}$. Observe that properties 4. and 5. guarantee that $\mathcal{A}$ is a mad family. Properties 1., 2. and 3. guarantee that $\mathcal{A}$ is strongly $\aleph_0$-separated. Hence, $\mathcal{A}$ is the desired family.
\end{proof}

\section{Questions and Remarks}

\noindent We don't even have consistent examples to answer the following questions: 

\begin{question} Is there a partly-normal not quasi-normal almost disjoint family?
\end{question}

\begin{question} Is there an almost-normal not normal almost disjoint family?
\end{question}

\begin{question} Is there an almost-normal mad family?
\end{question}

\noindent If $\mathcal{A}$ is mad, $\Psi(\mathcal{A})$ is a pseudocompact and not countably compact space. Recall that normal pseudocompact spaces are countably compact and so it is natural to ask the following more general question

\begin{question} Are almost-normal pseudocompact spaces countably compact?
\end{question}

\noindent Since $\Psi$-spaces are always Tychonoff and not countably compact, the existence of an almost-normal mad family would answer this question in the negative. 

\noindent Finally, we have not considered the relationship between these weakenings of normality and countable paracompactness: 

\begin{question} Is there a relationship between countably paracompact and any of  these weakenings of normality?

\end{question}

\section*{Acknowledgements}
The first author was partly supported for this research by the Consejo Nacional de Ciencia y Tecnolog\'ia CONACYT, M\'exico, Scholarship 411689.

\small

\textsc{Department of Mathematics and Statistics, York University, 4700 Keele St. Toronto, ON M3J 1P3 Canada}\par\nopagebreak
\textit{Email address}: S. A. Garcia-Balan: \texttt{sgarciab@yorku.ca}\\ Paul J. Szeptycki: \texttt{szeptyck@yorku.ca}


\baselineskip=5pt

\begin{thebibliography}{99}

\bibitem{AS} S. Arya, M. Singal \emph{Almost normal and almost completely regular spaces}, Kyungpook Math. J., Volume 25 ,1 (1970), 141-152.

\bibitem{E} R. Engelking, \emph{General Topology}, Heldermann Verlag, Berlin, Sigma Series in Pure Mathematics 6, 1989.

\bibitem{GS} F. Galvin, P. Simon, \emph{A \v{C}ech function in ZFC}, Fund. Math. 193 (2007), 181-188.

\bibitem{GJ} L. Gillman, M. Jerison, \emph{Rings of Continuous Functions}, Van Nostrand, Princeton, NJ. 1960.

\bibitem{HH} F. Hernández-Hernández, M. Hrušák, \emph{Topology of Mrówka-Isbell Spaces}, in: Hrušák M., Tamariz-Mascarúa Á., Tkachenko M. (eds) Pseudocompact Topological Spaces. Developments in Mathematics, vol 55. Springer, Cham (2018).

\bibitem{HS} M. Hru\v s\'ak, P. Simon, \emph{Completely separable MAD families}, in Open Problems in Topology II, Edited by Elliott Pearl, 2007, Elsevier.

\bibitem{K2002} L. Kalantan, \emph{Results about $\kappa$-normality}, Topol. Appl. 125, 1 (2002), 47-62.

\bibitem{K2008} L. Kalantan, \emph{$\pi$-normal topological spaces}, Filomat 22 (2008),  1, 173–181.

\bibitem{K2017} L. Kalantan, \emph{A Mr\'owka space which is quasi normal}, May 14, 2017. Preprint.

\bibitem{KS} L. Kalantan, P. Szeptycki, \emph{$k$-normality and products of ordinals}, Topology and its Applications, 123, 3 (2002) 537-545.

\bibitem{Mr} S. Mr\'owka, \emph{On completely regular spaces}, Fund. Math. 41 (1955) 105-106.

\bibitem{Lu} N. N. Luzin, \emph{On subsets of the series of natural numbers}, Izvestiya Akad. Nauk SSSR. Ser. Mat., 11 (1947), 403-410.

\bibitem{Shch} E. V. Shchepin, \emph{Real function and near normal spaces}, Sibirskii Mat. Zurnal, 13 (1972), 1182-1196.

\bibitem{SS} A.R. Singal, M. K .Singal, \emph{Mildly normal spaces}, Kyungpook Math. J.,
 13 (1973), 27-31.

\bibitem{Za} V. Zaitsev, \emph{On certain classes of topological spaces and their bicompactifications}, Dokl. Akad. Nauk SSSR, 178 (1968), 778-779.

\end{thebibliography}
\end{document}